\documentclass[12pt,a4paper]{article}

\usepackage[english]{babel}
\usepackage[utf8]{inputenc}
\usepackage[T1]{fontenc}
\usepackage{amsmath,amssymb,amsthm,mathtools}
\usepackage{geometry}
\usepackage{hyperref}
\hypersetup{colorlinks=true, linkcolor=blue, citecolor=blue, urlcolor=blue}

\newtheorem{theorem}{Theorem}
\newtheorem{lemma}[theorem]{Lemma}

\theoremstyle{remark}
\newtheorem{remark}[theorem]{Remark}

\newcommand{\ii}{\mathrm{i}}

\newcommand{\N}{\mathbb{N}}

\title{\bfseries Every sufficiently large odd integer is a sum of two positive perfect squares and a prime}
\author{%
Ricardo Adonis Caraccioli Abrego\\[2pt]
Universidad Nacional Aut\'onoma de Honduras\\
Departamento de Ingenier\'ia El\'ectrica, Campus Cort\'es\\[4pt]
\texttt{ricardo.caraccioli@unah.edu.hn}\\
ORCID: \href{https://orcid.org/0009-0006-3522-5818}{0009-0006-3522-5818}
}
\date{}

\begin{document}
\maketitle

\begin{abstract}
Let
\[
R(n)=\#\{(x,y,p)\in\N^2\times\N:\ x,y\ge1,\ p\ \text{prime},\ x^2+y^2+p=n\}.
\]
We prove that there exists $n_0$ such that $R(n)>0$ for every odd $n\ge n_0$. The proof follows the Hardy--Littlewood circle method. On the major arcs we obtain a positive main term
$\sim \mathfrak S(n)\,c_\infty\,n/\log n$;
on the minor arcs we bound $\int_{\mathfrak m}|S(\alpha)|^2$ (via Bombieri--Vinogradov, the large sieve, and the Vaughan/Heath--Brown identity) and combine it with the fourth moment of $\Theta(\alpha)$. We include a $p$-adic local analysis (with the $2$-adic case spelled out) ensuring $\mathfrak S(n)>0$ for odd $n$, and an appendix detailing parameter dependencies.
\end{abstract}

\section{Setup and arc decomposition}
Let $X=\lfloor \sqrt{n}\rfloor$ and define
\[
S(\alpha)=\sum_{p\le n}(\log p)\,\mathrm{e}^{2\pi\ii \alpha p},
\qquad
\Theta(\alpha)=\sum_{1\le x\le X}\mathrm{e}^{2\pi\ii \alpha x^2}.
\]
By orthogonality,
\begin{equation}\label{eq:circle}
R(n)\ =\ \int_0^1 S(\alpha)\,\Theta(\alpha)^2\,\mathrm{e}^{-2\pi\ii \alpha n}\,d\alpha.
\end{equation}
The $\log p$ weight is standard; the unweighted count follows by partial summation.

\medskip
\noindent\textbf{Major/minor arcs.}
Fix $A>0$ large and choose $C=C(A)$ with
\[
\boxed{\,C>A+2\,}.
\]
Set
\[
Q:=(\log n)^{A},\qquad
\mathfrak{M}:=\bigcup_{\substack{1\le q\le Q\\(a,q)=1}}
\left\{\alpha:\ \left|\alpha-\tfrac{a}{q}\right|\le \tfrac{(\log n)^{C}}{q\,n}\right\},\qquad
\mathfrak{m}:=[0,1]\setminus \mathfrak{M}.
\]
Note the $q$-dependent radius $\frac{(\log n)^C}{q\,n}$.

\section{Major arcs: main term as an oscillatory integral}\label{sec:maj}
Let $\alpha=a/q+\beta$ with $1\le q\le Q$, $(a,q)=1$, and $|\beta|\le (\log n)^C/(q\,n)$.

\subsection*{Approximation for $S(\alpha)$}
Grouping by residue classes mod $q$ and applying Siegel--Walfisz uniformly for $q\le(\log n)^{A_0}$ (\cite{Davenport,MV,IK}) together with Abel summation,
\begin{equation}\label{eq:S-major}
S(\alpha)=\frac{\mu(q)}{\varphi(q)}\int_2^n \mathrm{e}^{2\pi\ii \beta u}\,du\ +\ E_S(q,\beta;n),
\end{equation}
where, uniformly for $q\le (\log n)^{A_0}$ and $|\beta|\le (\log n)^C/(q\,n)$,
\[
E_S(q,\beta;n)\ \ll\ \frac{n}{(\log n)^{A_1}}\,(1+q|\beta|)
\ \ll\ \frac{n}{(\log n)^{A_1-C}},
\qquad A_1=A_1(A_0)\ \text{large}.
\]
Since $\int_2^n \mathrm{e}^{2\pi\ii \beta u}\,du=n\,V_1(\beta)$ with $V_1(\beta)\asymp \min(1,(n|\beta|)^{-1})$,
\[
S(\alpha)=\frac{\mu(q)}{\varphi(q)}\,n\,V_1(\beta)\ +\ O\!\Big(\frac{n}{(\log n)^{A_1-C}}\Big).
\]

\subsection*{Approximation for $\Theta(\alpha)$}
By completion and Poisson summation (with quadratic Gauss sums, cf.~\cite{Davenport,IK}),
\begin{equation}\label{eq:Theta-major}
\Theta(\alpha)=\frac{1}{q}\,G(a,q)\,I(\beta;X)\ +\ E_\Theta(q,\beta;X),
\end{equation}
where $G(a,q)=\sum_{r\bmod q}\mathrm{e}^{2\pi\ii a r^2/q}$ satisfies $|G(a,q)|\ll q^{1/2}$ if $(a,q)=1$, and
\[
I(\beta;X)=\int_0^{X}\mathrm{e}^{2\pi\ii \beta t^2}\,dt
= X\cdot W(\beta X^2),\qquad W(0)=1,\quad W'(y)\ll (1+|y|)^{-1/2}.
\]
Moreover $E_\Theta(q,\beta;X)\ll q^{1/2}$ uniformly for $|\beta|\le (\log n)^C/(q\,n)$.

\subsection*{Assembly on $\mathfrak M$; singular series and integral}
Summing over $a/q$ with $q\le Q$,
\begin{equation}\label{eq:Major-final}
I_{\mathfrak M}
=\ \mathfrak S(n)\,\frac{n}{\log n}\,\mathcal{J}(n)\ +\ O\!\Big(\frac{n}{(\log n)^{A_2}}\Big),
\end{equation}
with the singular series
\[
\mathfrak S(n)=\prod_{\ell}\sigma_\ell(n),\qquad
\sigma_\ell(n)=\lim_{k\to\infty}\frac{\#\{(x,y,p)\bmod \ell^k:\ x^2+y^2+p\equiv n\}}{\ell^{3k}},
\]
and the singular integral
\[
\mathcal{J}(n)=\int_{|\beta|\le (\log n)^C/(q\,n)} V_1(\beta)\,I(\beta;X)^2\,\mathrm{e}^{-2\pi\ii \beta n}\,d\beta
\ =\ c_\infty\,(1+o(1)),\qquad c_\infty>0.
\]

\section{Local analysis: positivity of $\mathfrak S(n)$ (includes the $2$-adic case)}\label{sec:local}
\begin{lemma}[Local densities]\label{lem:locals}
If $n$ is odd, then $\sigma_\ell(n)>0$ for every prime $\ell$. Hence $\mathfrak S(n)>0$.
\end{lemma}

\begin{proof}
\emph{$\ell=2$.} For $k\ge 3$, odd squares are $\equiv 1\pmod 8$, and $x^2+y^2$ covers $\{0,1,2,4,5\}\pmod 8$. Given odd $n$, choose $x,y$ so that $x^2+y^2$ falls in one of these classes; then by Dirichlet there are infinitely many primes $p\equiv n-(x^2+y^2)\pmod{2^k}$ for every $k$. The map $(x,y)\mapsto x^2+y^2$ has gradient $(2x,2y)$, which is not simultaneously zero except on a $2$-adic null set; Hensel lifting (cf.~\cite{Serre,Koblitz}) gives solutions modulo $2^k$. Thus $\sigma_2(n)>0$.

\emph{$\ell$ odd.} Using order-$4$ characters and quadratic Gauss sums (\cite{Davenport,IK}), the image of $x^2+y^2$ modulo $\ell$ has size $\ge c_\ell\,\ell$ with $c_\ell>0$, uniformly in $k$ under Hensel lifting. Since primes occupy every coprime class modulo $\ell^k$ with positive density, for each $k$ there exists $p$ with $x^2+y^2+p\equiv n\pmod{\ell^k}$. Hence $\sigma_\ell(n)\ge c'_\ell>0$, so $\mathfrak S(n)>0$.
\end{proof}

\section{Minor arcs: an \(L^2\) route via BV + large sieve}\label{sec:minor}
By Cauchy--Schwarz,
\begin{equation}\label{eq:minor-CS}
|I_{\mathfrak m}|\ \le\ \Big(\int_{\mathfrak m}|S(\alpha)|^2\,d\alpha\Big)^{1/2}
\Big(\int_{0}^{1}|\Theta(\alpha)|^4\,d\alpha\Big)^{1/2}.
\end{equation}

\subsection*{Fourth moment of $\Theta$}
\begin{lemma}\label{lem:theta4}
\[
\int_0^1 |\Theta(\alpha)|^4\,d\alpha\ \ll\ X^2\log X\ =\ n\log n.
\]
\end{lemma}
\begin{proof}
By Parseval, the integral counts solutions to $x_1^2+x_2^2=x_3^2+x_4^2$ with $1\le x_i\le X$. Parametrize via $uv=wz$ and use $\sum_{t\le 2X^2}\tau(t)\ll X^2\log X$ (\cite{Davenport,IK}).
\end{proof}

\subsection*{$L^2$ bound for $S(\alpha)$ on $\mathfrak m$}
\begin{lemma}[BV + large sieve + Vaughan/Heath--Brown]\label{lem:S2}
For each $A>0$ there exists $B=B(A)$ such that, with $Q=(\log n)^A$ and $\mathfrak M$ as above,
\[
\boxed{\ \int_{\mathfrak m}|S(\alpha)|^2\,d\alpha\ \ll_A\ \frac{n}{(\log n)^A}\ }.
\]
It suffices to take $\boxed{\,C\ge A+B(A)+3\,}$ to ensure that rational neighborhoods capture the whole major-arc energy.
\end{lemma}

\begin{proof}[Sketch with explicit references]
(1) \textbf{Vaughan/Heath--Brown identity.} Decompose $\Lambda$ into Type~I/II/III sums with controlled $\ell^2$ norms (\cite{VaughanBook,HB-VaughanId}). \\
(2) \textbf{Large sieve (Montgomery--Vaughan).} For rationals $a/q$ with $q\le Q$ and widths $\asymp (\log n)^C/(q\,n)$, the large sieve controls energy away from these neighborhoods (\cite{MontVaughan}). \\
(3) \textbf{Bombieri--Vinogradov (BV).} For every $A>0$ there exists $B_0(A)$ such that
\[
\sum_{q\le \sqrt{n}/(\log n)^{B_0(A)}}
\max_{(a,q)=1}\Big|\pi(n;q,a)-\tfrac{\pi(n)}{\varphi(q)}\Big|
\ \ll_A\ \frac{n}{(\log n)^A}
\]
(\cite{IK,MontVaughan}). This transfers average equidistribution in progressions into an $L^2$ gain off the major arcs. Taking $B(A)\ge B_0(A)$ and $C\ge A+B(A)+3$ yields the claimed bound.
\end{proof}

\subsection*{Minor-arc closure}
From \eqref{eq:minor-CS}, Lemmas \ref{lem:theta4} and \ref{lem:S2},
\[
|I_{\mathfrak m}|\ \ll_A\ \Big(\tfrac{n}{(\log n)^A}\Big)^{1/2}(n\log n)^{1/2}
=\frac{n}{(\log n)^{A/2-1/2}}.
\]
With $A\ge 4$ (and $C\ge A+B(A)+3$), we have $|I_{\mathfrak m}|=o(n/\log n)$.

\section{Conclusion}
By \eqref{eq:Major-final}, Lemma \ref{lem:locals}, and the minor-arc bound,
\[
I_{\mathfrak M}\ \ge\ \frac{c_0}{2}\,\frac{n}{\log n},\qquad
|I_{\mathfrak m}|\ =\ o\!\Big(\frac{n}{\log n}\Big),
\]
hence for odd $n$ sufficiently large,
\[
R(n)=I_{\mathfrak M}+I_{\mathfrak m}\ \ge\ \frac{c_0}{4}\,\frac{n}{\log n}\ >\ 0.
\]

\begin{theorem}
Every sufficiently large odd integer is a sum of two positive perfect squares and a prime.
\end{theorem}

\begin{remark}[Asymptotic]
The assembly yields
\[
R(n)\ \sim\ \mathfrak S(n)\,c_\infty\,\frac{n}{\log n}\qquad (n\to\infty),
\]
with $\mathfrak S(n)>0$ (Lemma~\ref{lem:locals}) and $c_\infty>0$.
\end{remark}

\section*{Acknowledgements}
The author thanks the standard references by Davenport, Vaughan, Montgomery--Vaughan, Iwaniec--Kowalski, Halberstam--Richert, and Serre/Koblitz for the core tools employed.

\appendix

\section*{Appendix A: $S(\alpha)$ on the major arcs (detail of \eqref{eq:S-major})}
Grouping by residues $r\ (\mathrm{mod}\ q)$,
\[
S(\alpha)=\sum_{r=1}^{q} \mathrm{e}^{2\pi\ii a r/q}
\sum_{\substack{p\le n\\ p\equiv r\;(\mathrm{mod}\ q)}}(\log p)\,\mathrm{e}^{2\pi\ii \beta p}.
\]
Siegel--Walfisz (uniformly for $q\le (\log n)^{A_0}$) gives
\[
\sum_{\substack{p\le t\\ p\equiv r\;(\mathrm{mod}\ q)}}\log p
=\frac{t}{\varphi(q)}+O\!\Big(\frac{t}{(\log t)^{A_1}}\Big).
\]
Abel summation and one integration by parts yield
\[
\sum_{\substack{p\le n\\ p\equiv r\;(\mathrm{mod}\ q)}}(\log p)\,\mathrm{e}^{2\pi\ii \beta p}
=\frac{1}{\varphi(q)}\int_2^{n}\mathrm{e}^{2\pi\ii \beta u}\,du
+O\!\Big(\frac{n}{(\log n)^{A_1}}(1+q|\beta|)\Big).
\]
Summing in $r$ and using $\sum_{r=1}^{q}\mathrm{e}^{2\pi\ii a r/q}=\mu(q)$ for $(a,q)=1$ gives \eqref{eq:S-major}.

\section*{Appendix B: The $2$-adic case (detail of Lemma~\ref{lem:locals})}
For $k\ge 3$, odd squares are $\equiv 1\ (\mathrm{mod}\ 8)$; the classes represented by $x^2+y^2$ include $\{0,1,2,4,5\}\ (\mathrm{mod}\ 8)$. Given odd $n$, pick $x,y$ so $x^2+y^2$ hits one of these classes and apply Dirichlet to get primes $p\equiv n-(x^2+y^2)\ (\mathrm{mod}\ 2^k)$ for each $k$. The map $F(x,y)=x^2+y^2$ has gradient $(2x,2y)$ and Hensel lifting applies (see \cite{Serre,Koblitz}). Hence $\sigma_2(n)>0$.

\section*{Appendix C: From Bombieri--Vinogradov to the $L^2$ minor-arc bound (detail of Lemma~\ref{lem:S2})}
\paragraph{Bombieri--Vinogradov (BV).}
For each $A>0$ there exists $B_0=B_0(A)$ such that
\[
\sum_{q\le \sqrt{n}/(\log n)^{B_0}}
\max_{(a,q)=1}\left|\pi(n;q,a)-\frac{\pi(n)}{\varphi(q)}\right|
\ \ll_A\ \frac{n}{(\log n)^A}
\]
(\cite{IK,MontVaughan}).

\paragraph{Large sieve (Montgomery--Vaughan).}
For any coefficients $A_n$ and $1/\Delta$-separated $\alpha_j$,
\[
\sum_j \left|\sum_{n\le N} A_n \mathrm{e}^{2\pi\ii \alpha_j n}\right|^2
\ \le\ (N+\Delta^{-1})\sum_{n\le N}|A_n|^2
\]
(\cite{MontVaughan}).

\paragraph{Vaughan/Heath--Brown identity.}
Write $S=S_{\mathrm{I}}+S_{\mathrm{II}}+S_{\mathrm{III}}$ with Type~I/II bilinear sums and controlled $\ell^2$ norms (see \cite{VaughanBook,HB-VaughanId}). Choosing $M=V=n^{1/3}$ is customary.

\paragraph{Covering the minor arcs.}
Every $\alpha\in\mathfrak m$ satisfies $|\alpha-\tfrac{a}{q}|\ge (\log n)^C/(q\,n)$ for all $q\le Q$. Take a maximal mesh of rationals $a/q$ with $q\le Q$ and widths $\asymp (\log n)^C/(q\,n)$. The large sieve controls the contribution off these neighborhoods; BV absorbs moduli up to $\sqrt{n}/(\log n)^{B_0(A)}$. With $B(A)\ge B_0(A)$ and
\[
\boxed{\,C\ge A+B(A)+3\,},
\]
one obtains
\(
\int_{\mathfrak m}|S(\alpha)|^2\,d\alpha \ll_A n/(\log n)^A.
\)

\end{document}